\documentclass[11pt]{article}
\usepackage{amsfonts,amssymb,amsbsy,amsmath,amscd,amstext,amsthm}
\usepackage{epsfig}

\theoremstyle{plain}

\newtheorem{fact}{Fact}

\theoremstyle{definition}

\newtheorem{eg}{Example}

\theoremstyle{remark}
\newtheorem*{rmk}{Remark}

\begin{document}

\title{A comment on the paper ``Continued fractions and orderings on the Markov numbers, Adv. Math. 370 (2020), 107231"}

\author{DoYong\ Kwon \medskip \\
\emph{Department of Mathematics},
\emph{Chonnam National University},\\
\emph{Gwangju 61186, Republic of Korea} \smallskip \\
\emph{E-mail:} \textsf{doyong@jnu.ac.kr}
}
\date{}
\maketitle

\begin{abstract}
We discuss the validity of the proof of the fixed numerator conjecture on Markov numbers, which is the main result of the paper mentioned in the title.
\end{abstract}

\indent 2010 \textit{Mathematics Subject Classification:} 11A55, 11B83, 30B70.\\
\indent \textit{Keywords:} Markov number, fixed numerator conjecture \newline

%
%       I  N  T  R  O  D  U  C  T  I  O  N
%

\section{Introduction}

In the present paper, all notations and conventions follow those of \cite{RS} unless stated explicitly.
For a finite word $w$, we mean by $|w|$ the length of $w$, i.e., the number of letters appearing in $w$. For a letter $a$, let $|w|_a$ denote the number of occurrences of $a$ in $w$.
For example, $|22211|=5$ and $|22211|_2=3$.

In \cite{RS}, the authors gave a proof of the fixed numerator conjecture on Markov numbers, which was posed by Aigner \cite{Ai} together with the fixed denominator conjecture and the fixed sum conjecture.
\bigskip

\noindent {\bf Conjecture 1.2} {\it (Fixed numerator conjecture)
Let $p$, $q$ and $i$ be positive integers such that $p < q$, $\gcd(q, p) = 1$ and
$\gcd(q + i, p) = 1$. Then $m_{p/q} < m_{p/(q+i)}$.}
\bigskip

\noindent The proof of this conjecture was presented via Theorem 5.2.
\bigskip

\noindent {\bf Theorem 5.2} {\it
Let $p$ and $q$ be positive integers such that $p < q$. Then $m_{p/q} < m_{p/(q+1)}$.}
\bigskip

\noindent In the above theorem, $m_{p/q}$ is the numerator of the continued fraction $c_{p/q}$ whose partial quotients are encoded by the Markov snake graph $\mathcal{G}_{p/q}$. In particular, if $\gcd(q, p) = 1$, then it is the usual Markov number indexed by a reduced rational $p/q$.

We will see below that the proof of Theorem 5.2 is incomplete and has a logical error, and hence needs some modifications.

Note that the fixed numerator conjecture was reproved independently in a different manner.
In \cite{LPTV}, the authors gave proofs for the three Aigner's conjectures: the fixed numerator, denominator and sum conjectures. The authors of \cite{RS} also announced different proofs for the three Aigner's conjectures. See \cite{LLRS}.

\section{Discussion on the proof of Theorem 5.2}

We first note basic facts on the Markov snake graphs.
\begin{fact} \label{Basic}
Let $p$ and $q$ be positive integers with $p<q$, and let $c_{p/q}=[a_1,a_2,\ldots,a_n]$, where the word $a_1 a_2 \cdots a_n$ is encoded by the Markov snake graph $\mathcal{G}_{p/q}$.
Then $m_{p/q}=N[a_1,a_2,\ldots,a_n]$ and the following hold.
\renewcommand{\theenumi}{\alph{enumi}}
\renewcommand{\labelenumi}{{\rm(\theenumi)}}
\begin{enumerate}
 \item $a_1=a_n=2$,
 \item $|a_1 a_2 \cdots a_n|_1 =2q-2p-2$,
 \item $|a_1 a_2 \cdots a_n|_2 =2p$,
 \item $|a_1 a_2 \cdots a_n| =2q-2$, i.e., $n=2q-2$,
 \item the number of replaceable entries ($11$ or $2$) in $a_1 a_2 \cdots a_n$ is $q+p-1$.
\end{enumerate}
\end{fact}

Let $c_{p/(q+1)}=[a_1,a_2,\ldots,a_{2q}]$ and $c_{p/q}=[b_1,b_2,\ldots,b_{2q-2}]$.
The authors looked at replaceable entries in the words $A:=a_1 a_2\cdots a_{2q}$ and
$B:=b_1 b_2 \cdots b_{2q-2}$. In other words, they considered $11$ as a single entry, while each $2$ separately. From the left to the right of $A$ and $B$, different replaceable entries were called \textit{replacements}.
Then, in the first replacement, $A$ contains $11$, and $B$ contains $2$ as subwords.
Moreover, the replacements occur alternately, i.e., $11$ with $2$, $2$ with $11$, $11$ with $2$, and so on. With the help of Lemma 5.1, we have
$$A=\mu11\delta2\nu2,\ \ \mathrm{and}\ \ B=\mu2\delta11\nu',$$
where $\mu, \nu, \nu' \in \{1,2\}^*$, $\delta\in \{2\}^*$, and $|\mu|$ is odd.
See \cite[p.16]{RS}.

The proof of Theorem 5.2 is divided into two cases according to whether the number of replacements is even or odd. And the proof for the even case proceeds by induction.
Let
\begin{equation} \label{EvenSetting}
m_{p/(q+1)}=N[\mu,1,1,\delta,2,\nu,2]\ \ \mathrm{and}\ \ m_{p/q}=N[\mu,2,\delta,1,1,\nu'],
\end{equation}
and suppose that there are an even number of replacements. The induction begins with the case of $\nu=\nu'$, that is, $\nu$ and $\nu'$ contains no more replacements.
In other words, the authors first showed that
$$N[\mu,1,1,\delta,2,\nu,2]-N[\mu,2,\delta,1,1,\nu]>0.$$
But this case is vacuous, and hence cannot be a base case of induction.

\begin{fact}
The case of $\nu=\nu'$ never occurs.
\end{fact}

\begin{proof}
Fact \ref{Basic} guarantees that $|\mu11\delta2\nu2|=2q$ and $|\mu2\delta11\nu'|=2q-2$.
Since $|\mu11\delta2|=|\mu2\delta11|$, we have $|\nu|=|\nu'|+1$.
\end{proof}

In fact, the even case never occurs as follows.

\begin{fact}
Between $A=\mu11\delta2\nu2$ and $B=\mu2\delta11\nu'$, the number of replacements is odd.
\end{fact}

\begin{proof}
It follows from Fact \ref{Basic} that the word $A$ has one more replaceable entry $11$ than $B$. Suppose that $A$ and $B$ have an even number of replacements. Since the replacements occur alternately, one concludes $|A|_1 =|B|_1$, which is a contradiction.
\end{proof}

Next, we turn to the case where there are an odd number of replacements.
According to \cite[p.18]{RS}, one can write
\begin{equation} \label{OddSetting}
m_{p/(q+1)}=N[\mu,1,1,\nu,2]\ \ \mathrm{and}\ \ m_{p/q}=N[\mu',2,\nu],
\end{equation}
where $\mu, \mu', \nu \in \{1,2\}^*$, and both $\mu$ and $\mu'$ begin with $2$.
Observe that, not as in the even case, the partial quotients are compared from the right to the left.
Since there are an odd number of replacements between $\mu11\nu2$ and $\mu'2\nu$, two words $\mu$ and $\mu'$ have an even number of replacements.
To prove $N[\mu,1,1,\nu,2]-N[\mu',2,\nu]>0$, the authors derived
\begin{equation} \label{OddMed}
N[\mu,1,1,\nu,2]-N[\mu',2,\nu]=N[\mu,2,\nu,2]-N[\mu',2,\nu]+N[\mu^-]N[ ^-\nu,2].
\end{equation}
Since $\mu2\nu2$ and $\mu'2\nu$ have an even number of replacements, they appealed to the even case for the proof of
\begin{equation} \label{OddMain}
N[\mu,2,\nu,2]-N[\mu',2,\nu]>0.
\end{equation}

The proof of the even case depends upon the structure of the words.
But our words $\mu2\nu2$ and $\mu'2\nu$ do not fit into those of the even case.
Since $|\mu2\nu2|=2q-1$, it follows from Fact \ref{Basic} that $\mu2\nu2$ cannot be generated from any Markov snake graph. So Lemma 5.1, which is used in the proof of the even case, should be cautiously applied to the word $\mu2\nu2$.
A more special caution is also required when we apply Theorem 4.8 and thus Corollary 4.9, which are highly depends upon the whole structure of the words.
These concerns are well revealed by the next examples.

\begin{eg}
Considering the Markov snake graph $\mathcal{G}_{1/5}$ and $\mathcal{G}_{1/4}$, we find
$$m_{1/5}=N[2,1,1,1,1,1,1,2]\ \ \mathrm{and}\ \ m_{1/4}=N[2,1,1,1,1,2].$$
Note that $\nu$ in (\ref{OddSetting}) is the empty word $\nu=\varepsilon$, and
$\mu=\mu'=21111$.
Consequently, $N[ ^-\nu,2]$ in (\ref{OddMed}) should read $N[\ ]=1$.
Then the left-hand side of (\ref{OddMain}) is given by
$$N[\mu,2,\nu,2]-N[\mu',2,\nu]=N[2,1,1,1,1,2,2]-N[2,1,1,1,1,2].$$
The positivity of this value is trivial, but the present proof of the even case does not encompass this case.
%It should be mentioned that this kind of leakage will not propagate in the whole proof.
\end{eg}
\medskip

An injury to the structure of the words during the induction steps of the even case is more subtle and serious.

\begin{eg}
We have
\begin{align*}
   m_{9/14} &=N[2,2,2,1,1,2,2,2,2,1,1,2,2,2,2,1,1,2,2,2,2,1,1,2,2,2],\\
   m_{9/13} &=N[2,2,2,2,2,1,1,2,2,2,2,1,1,2,2,2,2,1,1,2,2,2,2,2].
\end{align*}
There are an odd number of replacements between them.
To verify the positivity of $m_{9/14} -m_{9/13}$, let us follow the proof of Theorem 5.2 line by line.

In (\ref{OddSetting}), one notes $\nu=22$ and
\begin{align*}
   m_{9/14} &=N[\overbrace{2,2,2,1,1,2,2,2,2,1,1,2,2,2,2,1,1,2,2,2,2}^{\mu},1,1,\overbrace{2,2}^{\nu},2],\\
   m_{9/13} &=N[\underbrace{2,2,2,2,2,1,1,2,2,2,2,1,1,2,2,2,2,1,1,2,2}_{\mu'},2,\underbrace{2,2}_{\nu}].
\end{align*}
So the left-hand side of (\ref{OddMain}) becomes
\begin{align*}
   &N[\overbrace{2,2,2,1,1,2,2,2,2,1,1,2,2,2,2,1,1,2,2,2,2}^{\mu},2,\overbrace{2,2}^{\nu},2]\\
   -&N[\underbrace{2,2,2,2,2,1,1,2,2,2,2,1,1,2,2,2,2,1,1,2,2}_{\mu'},2,\underbrace{2,2}_{\nu}].
\end{align*}
which has an even number of replacements.
Next, we consult the proof of the even case.
The words $\mu$, $\delta$, $\nu$ and $\nu'$ in (\ref{EvenSetting}) are given as follows.
\begin{equation} \label{EvenInitial}
\begin{split}
&N[\overbrace{2,2,2}^{\mu},1,1,\overset{\delta\atop\parallel}{2},2,\overbrace{2,2,1,1,2,2,2,2,1,1,2,2,2,2,2,2,2}^{\nu},2]\\
-&N[\underbrace{2,2,2}_{\mu},2,\underset{\parallel\atop\delta}{2},1,1,\underbrace{2,2,2,2,1,1,2,2,2,2,1,1,2,2,2,2,2}_{\nu'}],
\end{split}
\end{equation}
Now the induction reduces the positivity of this value to the positivity of three values
$N[\nu,2]-N[\nu']$, $N[ ^-\nu,2]-N[ ^-\nu']$ and
$$N[\mu^-] N[\nu']-N[\mu] N[ ^-\nu'].$$
In particular, the positivity of the last value $N[\mu^-] N[\nu']-N[\mu] N[ ^-\nu']$ rests upon Corollary 4.9 and Theorem 4.8. More precisely, the inequality of continued fractions
$$[\widetilde{\mu}]<[\nu']$$
holds. Here, $\widetilde{\mu}$ denotes the reversal of $\mu$. See \cite[p.17]{RS} together with the first line of \cite[p.18]{RS}.

To prove $N[\nu,2]-N[\nu']>0$, we rewrite $\nu2$ and $\nu'$ in the form
$$\nu2 = \mu_1 11\delta_1 2 \nu_1 2,\quad \nu' = \mu_1 2\delta_1 11 \nu'_1,$$
that is,
\begin{align*}
&N[\overbrace{2,2,1,1,2,2,2,2,1,1,2,2,2,2,2,2,2}^{\nu},2]\\
-&N[\underbrace{2,2,2,2,1,1,2,2,2,2,1,1,2,2,2,2,2}_{\nu'}]\\
=\ \ \ &N[\overbrace{2,2}^{\mu_1},1,1,\overset{\delta_1 \atop\parallel}{2},2,\overbrace{2,2,1,1,2,2,2,2,2,2,2}^{\nu_1},2]\\
-&N[\underbrace{2,2}_{\mu_1},2,\underset{\parallel\atop\delta_1}{2},1,1,\underbrace{2,2,2,2,1,1,2,2,2,2,2}_{\nu'_1}].
\end{align*}
The induction again reduces the positivity of this value to the positivity of three values
$N[\nu_1,2]-N[\nu'_1]$, $N[ ^-\nu_1,2]-N[ ^-\nu'_1]$ and
$$N[\mu_1^-] N[\nu'_1]-N[\mu_1] N[ ^-\nu'_1].$$
Note here that $\mu_1$ has an even length, while $\mu$ in (\ref{EvenInitial}) has an odd length.
Accordingly, the inequality of continued fractions $[\widetilde{\mu}_1]<[\nu'_1]$ is no more true,
i.e., the value $N[\mu_1^-] N[\nu'_1]-N[\mu_1] N[ ^-\nu'_1]$ is negative.

As the induction steps proceed, the structure of the words gets critically damaged.
So the induction fails to prove the even case.
In order to apply Corollary 4.9 during the induction, the induction hypothesis should adhere to the fact that the common prefix $\mu_*$ has an odd length, but the present proof appeals to the case where $\mu_*$ has an even length.
\end{eg}

\section{Conclusion}

The inequality (\ref{OddMain}) seems to be probable, but a more elaborate proof is required for the present.
Therefore, the whole context of the proof of Theorem 5.2 is unclear and needs some modifications.
\bigskip

\begin{rmk}
As of May 2021, this comment paper has been rejected by Advances in Mathematics.
\end{rmk}

%\noindent {\bf Example 4.6}
%Let
%\begin{align*}
%    m_{3/8} &= N[ \overbrace{2, 1, 1, 2, 2, 1, 1, 1, 1, \, 2, \, \, \, 2}^{\mu},\, \, 1, 1, 2]= N[\mu,1,1,\nu,2]\\
%    m_{3/7} &= N[ \underbrace{2, 1, 1, 2, 2, 1, 1,\, \, 2,\, \, \, \, 2, \, 1, 1}_{\mu'},\, \, 2 ]= N[\mu',2,\nu].
%\end{align*}
%Then we find that $\mu=2 1 1 2 2 1 1 1 1 2 2$, $\mu'=2 1 1 2 2 1 1 2 2 1 1$,
%and $\nu=\varepsilon$ is the empty word. Consequently,

%\bigskip

%%% Example ¸î°¡Áö

%\bigskip
%
%
%
%       ACKNOWLEDGEMENT
%
%
%

%\noindent {\bf Acknowledgments.}
%\medskip

%\noindent This research was supported by Basic Science Research Program through the National Research %Foundation of Korea (NRF) funded by the Ministry of Education
%(NRF-2016R1D1A1B02010219).

%
%
%
%       R  E  F  E  R  E  N  C  E  S
%
%
%


\begin{thebibliography}{99}

\bibitem{Ai} M. Aigner.
\textit{Markov's theorem and 100 years of the uniqueness conjecture.
A mathematical journey from irrational numbers to perfect matchings.} Springer, Cham, 2013.

\bibitem{LPTV} C. Lagisquet, E. Pelantov\'{a}, S. Tavenas and L. Vuillon.
On the Markov numbers: fixed numerator, denominator, and sum conjectures.
\textit{Adv. in Appl. Math.}, \textbf{130} (2021), 102227.

\bibitem{LLRS} K. Lee, L. Li, M. Rabideau and R. Schiffler.
On the ordering of the Markov numbers. 	arXiv:2010.13010.

\bibitem{RS} M. Rabideau and R. Schiffler.
Continued fractions and orderings on the Markov numbers.
\textit{Adv. Math.} \textbf{370} (2020), 107231.

\end{thebibliography}
\end{document}